\documentclass[12pt]{amsart}

\usepackage{graphicx}

\usepackage{amsfonts}
\usepackage{amsmath,amssymb,amsthm,amssymb,amscd,cancel,stackrel}
\usepackage{blkarray}
\usepackage{multirow}
\usepackage{memhfixc}
\usepackage{latexsym}
\usepackage{tikz}
\usepackage{enumerate}
\usepackage{esint}

\newcommand{\fm}{\mathfrak m}

\newcommand{\mfa}{\mathfrak{a}}
\newcommand{\mfb}{\mathfrak{b}}
\newcommand{\mfc}{\mathfrak{c}}
\newcommand{\mfh}{\mathfrak{h}}

\newcommand{\ffi}{\varphi}
\newcommand{\al}{\alpha}
\newcommand{\be}{\beta}
\newcommand{\ga}{\gamma}
\newcommand{\la}{\lambda}

\newcommand{\om}{\omega}
\newcommand{\pa}{\partial}

\DeclareMathOperator{\coker}{coker}

\DeclareMathOperator{\Ext}{Ext}
\DeclareMathOperator{\grade}{grade}
\DeclareMathOperator{\Hom}{Hom}
\DeclareMathOperator{\id}{id}
\DeclareMathOperator{\Image}{Im}

\DeclareMathOperator{\pf}{Pf}

\DeclareMathOperator{\rank}{rank}
\DeclareMathOperator{\reg}{reg}

\DeclareMathOperator{\Tor}{Tor}

\newcommand*{\linkby[1]}{\thicksim_{#1}}
\newcommand{\sdeg}[1]{\mbox{\small $\deg{#1}$}}

\newtheorem{theorem}{Theorem}[section]
\newtheorem{lemma}[theorem]{Lemma}
\newtheorem{proposition}[theorem]{Proposition}
\newtheorem{corollary}[theorem]{Corollary}

\newtheorem{lem-def}[theorem]{Lemma and Definition}
\newtheorem{prop-def}[theorem]{Proposition and Definition}

\theoremstyle{definition}

\newtheorem{rem-def}[theorem]{Remark and Definition}
\newtheorem{example}[theorem]{Example}

\newcount\HOUR
\newcount\MINUTE
\newcount\HOURSINMINUTES
\newcount\INTVAL
\newcommand{\twodigit}[1]{\INTVAL=#1\relax\ifnum\INTVAL<10 0\fi\the\INTVAL}
\HOUR=\time\divide\HOUR by 60\relax
\HOURSINMINUTES=\HOUR\multiply\HOURSINMINUTES by 60\relax
\MINUTE=\time\advance\MINUTE by -\HOURSINMINUTES\relax
\newcommand\rightnow{
            \twodigit{\the\HOUR}:\twodigit{\the\MINUTE},
            \twodigit{\number\day}.\space
            \ifcase\month\or January\or February\or March\or April\or
May\or June\or July\or August\or September\or October\or November\or
December\fi
            \space\number\year}

\title{Constructing Homogeneous Gorenstein Ideals} 

\author{Sema G\"unt\"urk\"un}
\email{g.sema@uky.edu}
\address{Department of Mathematics,
University of Kentucky,
715 Patterson Office Tower, Lexington, KY 40506-0027\\
 USA}
\author{Uwe Nagel}
\email{uwe.nagel@uky.edu}
\address{Department of Mathematics,
University of Kentucky,
715 Patterson Office Tower, Lexington, KY 40506-0027\\
 USA}

\thanks{
    Part of the work for this paper was done while the  second author was  partially supported by the National Security Agency under Grant Number H98230-12-1-0247.}

\keywords{Gorenstein ideal, liaison, free resolution, elementary biliaison}
\subjclass[2010]{13H10, 13C40, 13D02, 14M12}

\begin{document}

\maketitle 

 \begin{abstract}
In 1983 Kustin and Miller introduced a construction of Gorenstein ideals in local Gorenstein rings, starting from smaller such ideals. We review and modify their construction
in the case of graded rings and  discuss it within the framework of  Gorenstein
liaison theory. We determine invariants of the constructed ideal. Concerning the problem of when a given Gorenstein ideal can be obtained by the construction, we derive  a necessary condition and exhibit a Gorenstein ideal that can not be obtained using the construction.
 \end{abstract}

\section{Introduction}
In \cite{KM} Kustin and Miller introduce a construction that produces, for given Gorenstein ideals
$\mfb \subset \mfa$  with grades $g$ and $g-1$, respectively, in a Gorenstein local ring $R$,   a new Gorenstein ideal $I$ of grade $g$ in a larger Gorenstein ring $R[v]$. Here $v$ is a new indeterminate. In \cite{KM2} they give an interpretation for their construction via liaison theory. These beautiful results prompted us to review their construction for homogeneous Gorenstein ideals in a graded Gorenstein ring. Instead of introducing a new indeterminate, we use  a suitable homogeneous element. The construction in \cite{KM} does not quite reveal the  conditions on that homogenous element. Therefore, we reverse the steps. We use two direct Gorenstein links to produce a new  Gorenstein ideal and to describe a generating set of it. Then we adapt the original Kustin-Miller construction suitably in order to produce a graded free resolution of the new Gorenstein ideal that is often minimal. We also consider the question of when the process can be reversed, that is, when can a Gorenstein ideal be obtained using the construction.

This paper is organized as follows. In Section 2 we recall some results from liaison theory and 
the mapping cone procedure.    In Section 3 we present a construction
of homogeneous Gorenstein ideals via liaison theory. Given two homogeneous Gorenstein ideals $\mfb \subset \mfa$ of grades of $g-1$ and $g$ in a graded Gorenstein ring $R$, by 
choosing an appropriate homogeneous element $f$ in $R$ we construct a homogeneous Gorenstein ideal $I = \mfb + (\al^*_{g-1} + (-1)^g f a_g^*)$ in the original ring $R$. Here $\al^*_{g-1}$ and $a_g^*$ are row vectors derived from comparing the resolutions of $\mfa$ and $\mfb$  and the second ideal is generated by the entries of the specified row vector (see Theorem~\ref{app_a_and_b}).

Using liaison theory, one also gets a graded free resolution of $I$. However, this resolution is never minimal. Adapting the original Kustin-Miller construction and its proof we obtain a smaller resolution that is often minimal (see Theorem~\ref{KM construction thm}). The key is a short exact sequence, which also allows us to interpret the linkage construction in Section 2 as an elementary biliaison from $\mfa$ on $\mfb$. Furthermore, we obtain a necessary condition on $\mfa$ for constructing a given Gorenstein ideal $I$ by such a biliaison (see Corollary~\ref{cor:necc cond}).

The original Kustin-Miller construction has been used to produce many interesting classes of Gorenstein ideals. In birational geometry it is known as unprojection (see, e.g., \cite{Pa, PR, BKR}). We illustrate the flexibility of our homogeneous construction by producing examples in Section 5. These include the Artinian Gorenstein ideals with socle degree two as classified by Sally \cite{sally} and the ideals of submaximal minors of a generic square matrix that are resolved by the Gulliksen-Neg\.{a}rd complex. We also consider some Tom unprojections as studied in \cite{BKR}.
We conclude with an example of a homogeneous Gorenstein ideal that can not be obtained using the construction of Theorem~\ref{app_a_and_b} with a strictly ascending biliaison.

\section{Liaison and mapping cones}

We frequently use ideas from liaison theory and mapping cones. We briefly recall some relevant concepts in this section.

Throughout this note $R$  denotes a commutative Noetherian ring that is either local with maximal ideal  $\fm$ or graded. In the latter case we assume that $R = \oplus_{j \ge 0} [R]_j$ is generated as $[R]_0$-algebra by $[R]_1$ and $[R]_0$ is a field. We denote by $\fm$ its maximal homogenous ideal $\oplus_{j \ge 0} [R]_j$. If $R$ is a graded ring, we consider only homogeneous ideals of $R$.

Assume that $R$ is a Gorenstein ring, that is, it has finite injective dimension as an $R$-module.  A {\em Gorenstein ideal} of $R$ is a perfect ideal $\mfc$  such that $R/\mfc$ is Gorenstein. An ideal $I \subset R$ is said to be (directly) \textit{linked} to an ideal $J \subset R$ by a Gorenstein ideal $\mfc \subset R$ if  $\mfc \subset I\cap J$  and $\mfc : I = J$ and $\mfc : J = I$. Symbolically, we write $I \linkby[\mfc] J$. Liaison is the equivalence relation generated by linkage. The equivalence classes are called liaison classes. Since we allow Gorenstein ideals in order to link, this is also referred to as Gorenstein liaison. We  always work in this generality. For a comprehensive introduction  to liaison theory we refer to \cite{Mi}.

It is not difficult to show that all complete intersections of a fixed grade are in the same liaison class. Much more is true.

\begin{theorem}
All Gorenstein ideals of $R$ of  grade $g$ are in the same liaison class.
\end{theorem}

This has been shown in \cite{CDH} for non-Artinian homogeneous Gorenstein ideals in a polynomial ring. However, the arguments work in this generality.

From now on we focus on graded rings as the results hold analogously for local rings if one forgets the grading.

Let $R$ be a graded Gorenstein ring, and let $M$ be a graded $R$-module. For any integer $s$, the module $M(s)$ is the module $M$ with the shifted grading given by $[M(s)]_j = [M]_{s+j}$.
 The $R$-dual of $M$ is the graded module $M^* = \Hom_R (M, R)$. The dual of $M$ with respect to the field $[R]_0$ is denoted by $M^{\vee}$. It is also a graded $R$-module.

The {\em canonical module} of $M$, denoted by $\omega_M$, is the $[R]_0$-dual of the local cohomology module $H^{\dim M}_{\fm} (M)$. By local duality, there is a graded isomorphism
$\omega_M \cong \Ext^g_R (M, R)(s)$, where $g = \dim R - \dim M$ and $s$ is the integer such that $\omega_R \cong R(s)$.

Now we recall a very useful short exact sequence. If the ideals $I$ and $J$ are linked by a Gorenstein ideal $\mfc$, there is a short exact sequence
\begin{equation}
     \label{eq:st-exact-seq}
0 \rightarrow \mfc \hookrightarrow I \rightarrow \om_{R/J}(-s) \rightarrow  0,
\end{equation}
where $s$ is the integer such that $\omega_{R/\mfc} \cong R/\mfc (s)$ (see, e.g., \cite[Lemma 3.5]{Na}).

For example, this sequence implies that $R/J$ is Cohen-Macaulay if $R/I$ has this property, and that the mapping cone procedure can be used to derive a free resolution of $\om_{R/J}$ from the resolutions of $I$ and $\mfc$, and thus of $J$ by dualizing  (see \cite{PS}). Because of its importance we recall the mapping cone procedure (see, e.g., \cite{W}).

\begin{lemma}
      \label{mapping cone}
Let
$$\begin{CD} 0 @>>> M @>\al>> N @>>> K @>>> 0 \end{CD}$$
be a short exact sequence of graded $R$-modules, and let
$$\minCDarrowwidth20pt\begin{CD}
\mathbb{F} : \  0 @>>>  F_n @>d^M_n>> F_{n-1}  @>>>    ... @>>>        F_i    @>d^M_i>>  ... @>>>     F_1   @>d^M_1>>    F_0  @>>> M @>>> 0
\end{CD}$$ and
$$\minCDarrowwidth20pt\begin{CD}
\mathbb{G} : \ 0 @>>>  G_n @>d^N_n>> G_{n-1}  @>>>    ... @>>>        G_i    @>d^N_i>>  ... @>>>     G_1   @>d^N_1>>    G_0  @>>> N @>>> 0
\end{CD}$$
be graded free resolutions of $M$ and $N$, respectively. Then  $\al$ induces a  comparison map  $\ffi : \mathbb{F} \to \mathbb{G}$. Its {\em mapping cone} is the following graded free resolution of $K$:
$$\minCDarrowwidth20pt\begin{CD}
  0 @>>>  F_n @>\pa_g>>   G_{n} \oplus F_{n-1} @>>>    ... @>>>    G_{i} \oplus  F_{i-1}  @>\pa_i>>  ... \hspace*{1cm}
\end{CD}$$
$$\minCDarrowwidth20pt\begin{CD}
\hspace*{5cm}  @>>>  G_1 \oplus  F_0    @>\pa_1>>    G_0  @>>> K @>>> 0,
\end{CD}$$
where  $\pa_i = \begin{bmatrix} d^N_{i} & \ffi_{i-1} \\[.5ex]  0 & -d^M_{i-1} \end{bmatrix}$ for $i = 1,2,...,n$.
\end{lemma}

An analysis of the mapping cone procedure implies the following result by  Buchsbaum-Eisenbud \cite{BE}  and Peskine-Szpiro  \cite{PS}.

\begin{lemma}\label{linkage of Gorenstein}
Let $\mfc$ be a Gorenstein ideal of $R$.Then
  \begin{itemize}
    \item[(a)] If $R/I$ is Gorenstein and $\mfc \subsetneqq I$ with $\grade(\mfc) = \grade(I)$, then  $J = \mfc : I$ is perfect with at most one more
	  minimal generator than $\mfc$.
    \item[(b)] Let $J \subset R$ be a perfect ideal such that $\mfc \subsetneqq J$, $\grade(\mfc) = \grade(J)$, and all minimal generators of $\mfc$ are also  minimal
          generators of $J$. If $J$ has one more minimal generator than $\mfc$, then  $I = \mfc : J$ is a Gorenstein ideal.
   \end{itemize}
\end{lemma}

In Case (b), if $\mfc$ is a complete intersection, then $J$ is an almost complete intersection, that is, $I$ has $g + 1$ minimal generators, where $g = \grade I$.


\section{Gorenstein Ideals obtained by two links}
        \label{const.via links}

In this section we use liaison to produce a homogeneous Gorenstein ideal starting from two given homogeneous Gorenstein ideals. This also allows us to relate the Hilbert functions of the involved ideals.

Let $R$ be a graded Gorenstein ring.
Let $\mfa$ and $\mfb \subset \mfa$ be homogeneous Gorenstein ideals in $R$ of grade $g$ and $g-1$, respectively.
Let
\begin{eqnarray*}
\begin{CD}
  \mathbb{A}  :  0 @>>> A_{g}=R(-v)  @>a_{g}>>   A_{g-1}        @>a_{g-1}>>  .... @>>>  A_{1}  @>a_{1}>>   R  @>>> 0
 \end{CD}
\end{eqnarray*}
and
\begin{eqnarray*}
\begin{CD}
  \mathbb{B}  :      0   @>>>        B_{g-1}= R(-u)  @>b_{g-1}>> .... @>>>  B_{1}  @>b_{1}>>   R  @>>> 0
 \end{CD}
\end{eqnarray*}
be graded minimal free resolutions of $R/\mfa$ and $R/\mfb$ respectively. The embedding $\mfb \hookrightarrow \mfa$ induces the following commutative diagram:
\begin{align} \label{comp_BtoA}
\minCDarrowwidth23pt \begin{CD}
  @.              0   @>>>        B_{g-1}= R(-u)  @>b_{g-1}>> .... @>>>  B_{1}  @>b_{1}>>   R  @>>> 0 \\
  @.      @ VVV                     @VV\alpha_{g-1}V            @.         @VV\alpha_{1}V     @VV\alpha_{0}= \id V\\
  0 @>>> A_{g}=R(-v)  @>a_{g}>>   A_{g-1}        @>a_{g-1}>>  .... @>>>  A_{1}  @>a_{1}>>   R  @>>> 0
 \end{CD}
\end{align}
Fixing bases for all the free modules, we identify the maps with their coordinate matrices.
Using these assumptions and notation, the main result of this section is:

\begin{theorem}
      \label{app_a_and_b}
Assume $d = u-v \geq 0$. Let $y \in \mfa$ be a homogeneous element such that $\mfb : y = \mfb$. The embedding $\mu: (\mfb, y) \hookrightarrow \mfa$ induces  an $R$-module homomorphism $\om_{R/\mfa} \to \om_{R/(\mfb, y)}$ that is multiplication by some homogeneous element $\om \in R$. Its degree is $d + \deg y$.

Assume there is a homogeneous element $f \in R$ of degree $d$ such that $\mfb : (\om+fy) = \mfb$. Consider  the ideal $I$ obtained from $\mfa$ by the two links
\[
\mfa \linkby[(\mfb, y)] J \linkby[(\mfb,\om+fy)] I,
\]
that is, $I = (\mfb,\om+fy) : [(\mfb, y) : \mfa]$. Then $I$
is a Gorenstein ideal with the same grade as $\mfa$. It can be written as
\[
I = \mfb + (\al^*_{g-1} + (-1)^g f a_g^*) =  (\mfb, \al^*_{g-1} + (-1)^g f a_g^*),
\]
where $\al^*_{g-1}$ and $a_g^*$ are interpreted as row vectors and
$``+``$ indicates their component-wise sum whose entries, together with generators of $\mfb$, generate $I$.
\end{theorem}

\begin{proof} As in the proof of the Lemma \ref{linkage of Gorenstein}, we use the mapping cone procedure repeatedly.
Multiplication by $y$ induces a short exact sequence
\[
0 \to R/\mfb (- e) \to R/\mfb \to R/(\mfb, y) \to 0,
\]
where $e = \deg y$.
Thus, we  obtain a minimal graded free resolution $\mathbb{B'}$ of  $(\mfb, y)$:
\begin{align*}
\small \minCDarrowwidth10pt \begin{CD}
\mathbb{B'} \ : 0 @>>> R(\mbox{\small $-u-e$}) @>d_{g}>> \begin{matrix}R(-u) \\ \oplus \\ B_{g-2}(-e) \end{matrix} @>>> ... @>>> \begin{matrix}B_1\\ \oplus \\ R(-e) \end{matrix} @>d_1>> (\mfb,y)  @>>> 0,
\end{CD}
\end{align*}
where
\[
d_1 = \begin{bmatrix} b_1& y \end{bmatrix},  \; d_g = \small \begin{bmatrix} (-1)^{g-1}y \\ b_{g-1}\end{bmatrix},  \; \text{ and }
d_i = \small \begin{bmatrix} b_i & (-1)^{i-1}yI_{m_{i-1}}\\ 0 & b_{i-1} \end{bmatrix} \text{ if } 2 \le i < g.
\]
Here $I_{m_i}$ denotes the identity matrix with $m_i = \rank B_i$ rows.

Using this resolution, the embedding  $\mu : (\mfb, y)  \hookrightarrow \mfa$ induces the following commutative diagram
\begin{align}
         \label{first comp map}
\minCDarrowwidth10pt\begin{CD}
  0 @>>>  R(\mbox{\small $-u-e$}) @>d_g>> \begin{matrix}R(-u) \\ \oplus \\ B_{g-2}(-e) \end{matrix}   @>>>    ... @>>> \begin{matrix}B_i \\ \oplus \\ B_{i-1}(-e) \end{matrix}  @>d_i>>  ... @>>> \begin{matrix}B_1\\ \oplus \\ R(-e) \end{matrix} @>d_1>> (\mfb, y)  @>>> 0\\
  @.      @V\mu_gVV  @V\mu_{g-1}VV       @.      @V\mu_{i}VV    @.     @V\mu_1VV  @V{\mu}VV \\
  0 @>>>  R(-v) @>a_{g}>>   A_{g-1}  @>>>  ...@>>>   A_{i} @>a_{i}>>        ... @>>>  A_{1}@>a_{1}>> \mfa @>>> 0, \\
\end{CD}
\end{align}
where the maps have the form
\[
\mu_i = \begin{bmatrix} \al_{i} &  r_i \end{bmatrix}.
\]
It shows that the map $\mu_g$ is multiplication by an element of degree $d+e$. By assumption, we may assume that $\om$ is this element.

The mapping cone $C(\mu)$ of $\mu : \mathbb{B'} \to \mathbb{A}$ is
\begin{eqnarray*}
\small \minCDarrowwidth10pt \begin{CD}
0 @>>> R(\mbox{\small $-u-e$}) @>\pa_g>> \begin{matrix}R(-v)\\ \oplus \\ R(-u) \\ \oplus \\ B_{g-2}(-e) \end{matrix} @>>> ...
@>>> \begin{matrix}A_{i+1} \\ \oplus \\B_i \\ \oplus \\ B_{i-1}(-e) \end{matrix}  @>\pa_i>>  ...@>>> \begin{matrix}A_2 \\ \oplus \\B_1\\ \oplus \\ R(-e) \end{matrix} @>\pa_1>> A_1,
\end{CD}
\end{eqnarray*}
where the maps are
$$\pa_1 = \begin{bmatrix}a_2 & \mu_1\end{bmatrix} = \begin{bmatrix}a_2 & \al_1 & r_1\end{bmatrix},\, \, \,
\pa_g = \begin{bmatrix}\mu_g \\ -d_g \end{bmatrix} = \begin{bmatrix}r_g \\ (-1)^gy \\ -b_{g-1} \end{bmatrix},$$
$$\mbox{and} \, \, \, \pa_i = \begin{bmatrix}a_{i+1} & \mu_i\\0&-d_i\end{bmatrix} =
\begin{bmatrix}a_{i+1} & \al_i & r_i \\ 0 & -b_i & (-1)^{i}yI_{m_{i-1}} \\ 0 & 0 & -b_{i-1} \end{bmatrix} \text{ if } 2 \le i < g.
$$
The Sequence \eqref{eq:st-exact-seq} shows that $C(\mu)$ gives a free resolution of a shift of the canonical module of $R/J$. Hence, the dualized and shifted complex $C(\mu)^*(-u-e)$ provides a graded free resolution of the ideal $J =(\mfb, y, \om)$. We do not claim that the stated generating set of $J$ is minimal.


By assumption, there is a homogeneous element $f\in R$ of degree $u-v = d \ge 0$ such that $z := \om+fy$ is regular in $R/\mfb$. Hence, $(\mfb, z)$ is a Gorenstein ideal of grade $g$ in $J$.
Consider now the second link
$$
J \linkby[(\mfb, \om+fy)] I.
$$
As in the case of the ideal $(\mfb, y)$, a mapping cone gives a free resolution of $(\mfb, z)$. Thus, the embedding $\xi: (\mfb, z) \hookrightarrow J$ induces the following commutative diagram:
\begin{eqnarray*}
\minCDarrowwidth10pt\begin{CD}
  0 @>>>  R(\mbox{\small $-u-e-d$}) @>t_g>> \mbox{\small $\begin{matrix}R(-u) \\ \oplus \\ B_{g-2}(-e-d) \end{matrix}$ }  @>>>   \ldots @>>> \mbox{\small $\begin{matrix}B_1\\ \oplus \\ R(-e-d) \end{matrix}$} @>t_1>> (\mfb, z) @>>> 0\\
  @.      @V\xi_gVV  @V\xi_{g-1}VV       @.      @V\xi_{1}VV         @V{\xi}VV   \\
  0 @>>> A_{1}^*(-u-e)@>\pa^*_1>>\mbox{\small $\begin{matrix}A_{2}^*(-u-e) \\ \oplus \\B_1^*(-u-e) \\ \oplus \\ R (-u) \end{matrix}$} @>>> \ldots @>>>
\mbox{ \small $\begin{matrix} A^*_g(-u-e)\\ \oplus \\ B^*_{g-1}(-u-e) \\ \oplus \\ B^*_{g-2}(-u-e)\end{matrix}$} @>\pa^*_g>> J @>>> 0
\end{CD}
\end{eqnarray*}
where the maps are
$$t_1 = \begin{bmatrix} b_1 & z\end{bmatrix}, \; t_g = \begin{bmatrix} (-1)^{g-1}z \\ b_{g-1}\end{bmatrix}, \; \text{ and }
t_i = \begin{bmatrix} b_i & (-1)^{i-1} z I_{m_{i-1}} \\ 0 & b_{i-1}\end{bmatrix} \text{ if }  2 \le i < g.
$$
Since $J = (\mfb, y, \om + f y)$, we can choose the following coordinate matrix for  $\xi_1$:
\[
\xi_1 =
 \begin{blockarray}{cccccc}
    & 1 & \dots  & \dots  & m_{1} &  m_1+1 \\
    \begin{block}{c[cccc@{\hspace*{4pt}}|c@{\hspace*{2pt}}]}
    1 & \BAmulticolumn{4}{c|}{\multirow{1}{*}{$0$}}       & 1\\
    2 & \BAmulticolumn{4}{c|}{\multirow{1}{*}{$0$}}       &(-1)^gf\\
    \cline{2-6}
    3 & \BAmulticolumn{4}{c|}{\multirow{4}{*}{\mbox{\huge $\ga_1$}}}& 0\\
    \vdots&  &&& & \vdots \\
    \vdots&  &&& & \vdots \\
    m_{1}+2 & &&& & 0 \\
    \end{block}
  \end{blockarray}
\]
where the matrix $\ga_1$ is  invertible.

By Sequence \eqref{eq:st-exact-seq}, the mapping cone $C(\xi)$  gives a free resolution of (a shift of)  the canonical module $R/I$. Using the self-duality of the free resolutions $\mathbb{A}$ and $\mathbb{B}$, $C(\xi)$ can be re-written as
\begin{eqnarray}\label{lastMC}
\minCDarrowwidth10pt \begin{CD}
0 @>>> B_{g-1}(-\sdeg{z}) @>l_g>> \small \begin{matrix}R(-u)\\ \oplus \\ B_{g-2}(-\sdeg{z}) \\ \oplus \\ A_{g-1}(-\sdeg{z}) \end{matrix} @>>> ... @>>>
\begin{matrix}B_1 \\ \oplus \\R(-\sdeg{z}) \\ \oplus \\ A_1(-\sdeg{z}) \\ \oplus \\ B_1(-\sdeg{y}) \\ \oplus \\ B_2 \end{matrix} @>l_1>>
\begin{matrix} R(-\sdeg{z}) \\ \oplus \\ R(-\sdeg{y}) \\ \oplus \\ B_1\end{matrix},  
\end{CD}
\end{eqnarray}
where
\[
l_1 = \begin{bmatrix}\pa^*_g & \xi_1\end{bmatrix}.
\]
Since the matrix $\gamma_1$ and the upper right entry of $\xi_1$ are invertible, the cokernel of $l_1$ is isomorphic to $\coker \bar{l}_1$, where
\[
\bar{l}_1 = \begin{bmatrix} \al^*_{g-1}+(-1)^gfa^*_g & b^*_{g-1}\end{bmatrix}: A_1 (- \deg z) \oplus B_1 (- \deg y) \oplus B_2 \to R (- \deg y).
\]
It follows that the canonical module of $R/I$ has only one minimal generator. Hence, $I$ is a Gorenstein ideal and $\coker \bar{l}_1 \cong (R/I) (-\deg y)$. The latter implies the claimed description of a generating set of the ideal $I$.
\end{proof}

Notice that a sufficiently general choice of the element $f$ always gives a desired element $\om + f y$ in Theorem~\ref{app_a_and_b}, at least if the field $k = R/\fm$ is infinite.

We illustrate the result by a simple example.

\begin{example}
    \label{ex:two-small-ci}
Consider the complete intersections $\mfa = (x, y, z)$ and $\mfb = (x^2 - z^2, y^2 - z^2)$ in the polynomial ring $K[x, y, z]$, where $K$ is a field of characteristic zero.     Linking $\mfa$ by $\mfb + (z^2)$, we get as residual $J = \mfb + (z^2, x y z)$. Choosing $f = 5 z$, we link $J$ by $\mfb + (x y z + f z^2)$ to
\[
I = \mfb + (x f + yz, y f + xz, z f + x y) = (x^2 - z^2, y^2 - z^2, x z, y z, x y + 5 z^2).
\]
Observe that for the second link we cannot take $f = z$ because $xyz + z^3$ is a zero divisor modulo $\mfb$.
\end{example}

Using basic properties of links, we conclude this section by relating the Hilbert function of $I$ to the Hilbert functions of $\mfa$ and $\mfb$. Recall that the Hilbert function of a graded $R$-module $M$ is defined by
$\mfh_M (j) = \dim_k [M]_j$.


\begin{corollary}
     \label{hilbert functions}
Adopt the notation and assumptions of Theorem~\ref{app_a_and_b}. Then, for all integers $j$, the Hilbert function of $R/I$ is given by
\[
\mfh_{R/I}(j) = \mfh_{R/\mfa} (j-d) +  \mfh_{R/\mfb}(j) - \mfh_{R/\mfb}(j-d).
\]
 \end{corollary}

\begin{proof}
In the proof of Theorem \ref{app_a_and_b} we have seen that the  mapping cone \eqref{lastMC}  gives the following short exact sequence;
$$\begin{CD}
 0 @>>>  (R/I) (-e)  @>>> R/(\mfb, z) @>>>  R/J  @>>>  0,
 \end{CD}
$$
where $\deg z = d+e$. Furthermore, by symmetry of liaison, the first link provides $(\mfb,y) : \om = \mfa$. This implies  the short exact sequence
$$
\begin{CD}   0 @>>> (R/\mfa)(-e-d) @>>>  R/(\mfb,y) @>>> R/J  @>>> 0  \end{CD}.
$$
Combining the above two sequences we deduce
\begin{align*}
 \mfh_{R/I}(j) & = \mfh_{R/(\mfb, z)} (j+e) - \mfh_{R/J} (j+e) \\
 & =  \mfh_{R/(\mfb, z)} (j+e) - \mfh_{R/(\mfb, y)} (j+e) + \mfh_{R/\mfa} (j-d) \\
 & = - \mfh_{R/\mfb} (j-d) + \mfh_{R/\mfb} (j) + \mfh_{R/\mfa} (j-d),
\end{align*}
as claimed.
\end{proof}


\section{A variation of the Kustin-Miller construction}\label{KM construction}

In the previous section we have seen that the Complex \eqref{lastMC} provides a free resolution of the Gorenstein ideal $I$, constructed in Theorem \ref{app_a_and_b}. However, this resolution is not minimal if $g \ge 3$. In this Section we construct a smaller resolution of $I$ by modifying the approach of Kustin and Miller in \cite{KM}.

\begin{theorem}
       \label{KM construction thm}
Adopt the notation and assumptions of Theorem~\ref{app_a_and_b}.
Then there is an short exact sequence of graded $R$-modules
$$\begin{CD}
  0 @>>> (\mfa/\mfb)(-d)@>>>   R/\mfb @>>> R/I @>>>  0.
 \end{CD}$$
 Moreover, the ideal $I$ has a graded free resolution of the form
 \begin{align*}
\minCDarrowwidth10pt \begin{CD}
 0 @>>> B_{g-1} (-d) @>>> \begin{matrix} A_{g-1} (-d) \\ \oplus \\ B_{g-2} (-d) \end{matrix} @>>> \begin{matrix} B_{g-2}  \\ \oplus \\ A_{g-2} (-d) \\ \oplus \\ B_{g-3} (-d) \end{matrix} @>>> \ldots  @>>> \begin{matrix} B_2 \\ \oplus \\ A_2 (-d) \\ \oplus \\ B_1 (-d) \end{matrix}
@>>> \begin{matrix}B_1 \\ \oplus \\ A_1 (-d) \end{matrix} @>>>  I  @>>> 0,
 \end{CD}
\end{align*}
where the maps are described in the proof below.
\end{theorem}

\begin{proof} We follow the approach in \cite{KM}, but  adjust it suitably. Thus, we focus on the needed modifications and refer for more details to \cite{KM}.

First, the mapping cone $\mathbb{M}$ of $\alpha : \mathbb{B}\rightarrow \mathbb{A}$ gives the exact sequence:
\begin{align}\label{mapp.cone M}
\minCDarrowwidth10pt\begin{CD}
\mathbb{M} : \ 0 @>>> \begin{matrix} A_g \\ \oplus \\ B_{g-1} \end{matrix} @>>>  \ldots @>>> \begin{matrix} A_{j+1}\\ \oplus \\ B_j \end{matrix} @>{\begin{bmatrix} a_{j+1} & \alpha_j \\ 0 & -b_j \end{bmatrix}}>>  \begin{matrix} A_{j+1}\\ \oplus \\ B_j \end{matrix}  @>>> \ldots  @>>> \begin{matrix} A_2 \\ \oplus  \\ B_1 \end{matrix}  @>{\begin{bmatrix}a_2 &\alpha_1\end{bmatrix}}>> A_{1} @>>> \mfa/\mfb.
\end{CD}
\end{align}

Second, by \cite[Proposition 1.1]{BE}, the resolutions $\mathbb{A}$ and $\mathbb{B}$ admit a DGC algebra structure. These induce perfect pairings $B_i \times B_{g-1-i} \to B_g$ and $A_i \times A_{g-i} \to A_g$. We use the former to define degree $d$ homomorphisms $\be_i: A_i \to B_{i-1} (d)$ by mapping $x_i \in A_i$ on the unique element $\be_i (x_i)$ such that, for all $z_{g-i} \in B_{g-i}$,
\begin{equation}
    \label{eq:def beta}
\be_i (x_i) \cdot z_{g - i} = (-1)^{i+1} x_i \cdot \al_{g-i} (z_{g-i})
\end{equation}
in $A_g = B_{g-1}(d)$. It follows that
\begin{equation}
   \label{eq:def-beta-1}
\be_1 (x_1) = x_1 \cdot \al_{g-1} (1_{B_{g-1}})
\end{equation}
and that $\be_{g}$ is multiplication by the unit $(-1)^{g+1}$.
Using the perfect pairings on $\mathbb{A}$, we also get
\begin{equation}
     \label{eq:beta-comm}
  \be_i \circ a_{i+1} = b_i \circ \be_{i+1}.
\end{equation}

Third, \cite[Lemma 1.1]{KM} shows that Diagram \eqref{comp_BtoA} induces a graded homomorphism of complexes $\xi : \mathbb{B} \otimes \mathbb{B}\rightarrow\mathbb{A}[1]$ such that, for all $z_i \in B_i$:
\begin{itemize}
 \item [(i)] $B_i\otimes B_j\rightarrow A_{i+j+1}$ is defined if $i,j\geq 0$,
 \item [(ii)] $\xi(z_i\otimes z_j)=(-1)^i\xi(z_j\otimes z_i)$,
 \item [(iii)] $\xi(z_i\otimes z_i) = 0$ if $i$ is odd,
 \item [(iv)] $\xi(z_0\otimes z_i) = 0$, \text{  and}
 \item [(v)] $\alpha_{i+j}(z_iz_j)- \alpha_i (z_i) \cdot \alpha_j (z_j) = \xi(b_i (z_i) \otimes z_j) + (-1)^{i}\xi(z_i\otimes b_j( z_j)) + a_{i+j+1} (\xi(z_i\otimes z_j))$.
 \end{itemize}

Finally, we  define a degree $d$ homomorphism of complexes  $h:\mathbb{B} \rightarrow \mathbb{B}(d)$ by mapping $z_i \in B_i$ on the   unique element $h_i(z_i)$ such that, for all  $z_{g-1-i}\in B_{g-1-i}$,
\[
h_i(z_i) \cdot z_{g-1-i}=(-1)^{i+1}\xi(z_i\otimes z_{g-1-i}).
\]
Notice that the above Condition (iv) implies $h_0 = h_{g-1} = 0$ and that Condition (v) yields
\begin{equation}
  \label{eq:homotopy-cond}
  \be_i \circ \al_i = h_{i-1} \circ b_i + b_i \circ h_i.
\end{equation}


Consider now the following diagram with exact rows $\mathbb{M}$ and $\mathbb{B}$, respectively:
\begin{align}\label{last mapp.cone}
\minCDarrowwidth10pt\begin{CD}
  0 @>>>  A_g\oplus B_{g-1} @>>>   ....     @>>> A_{j+1}\oplus B_j @>>>  .... @>>> A_2\oplus B_1  @>[a_2, \alpha_1]>> A_{1} @>>> \mfa/\mfb \\
  @.  @VV[\beta_g,(-1)^{g}f \id]V   @.         @VV[\beta_{j+1}, h_j +(-1)^{j+1}f \id]V    @.       @VV[\beta_2, h_1 +f \id]V  @VV\beta_1+f a_1V  \\
  0 @>>>  B_{g-1}(d) @>>>  ....   @>>> B_j(d)  @>b_{j}>>        .... @>>>  B_1(d) @>b_{1}>> R(d) @>>> (R/\mfb)(d)
\end{CD}
\end{align}
Using Equations \eqref{eq:beta-comm} and \eqref{eq:homotopy-cond}, we see that all the squares commute. It follows that $\beta_1+f a_1$ induces a homomorphism $\ffi: \mfa/\mfb \to (R/\mfb) (d)$ such that  the resulting right-most square in the above diagram also becomes commutative.  Thus, the mapping cone  gives the chain complex
\begin{align}\label{L}
\minCDarrowwidth20pt \begin{CD}
\mathbb{L} :\  @. 0 @>>> \begin{matrix} A_g\\ \oplus \\ B_{g-1}\end{matrix} @>{l_{g}}>> \begin{matrix}B_{g-1} (d)\\ \oplus \\ A_{g-1} \\ \oplus \\ B_{g-2} \end{matrix} @>>> ... @>>> \begin{matrix} B_2(d) \\ \oplus \\ A_2 \\ \oplus \\ B_1 \end{matrix}
@>{l_2}>> \begin{matrix}B_1(d) \\ \oplus \\ A_1 \end{matrix} @>{l_1}>>  R(d),
 \end{CD}
\end{align}
where the maps are
$$
l_1= \begin{bmatrix} b_1 & \beta_1+fa_1 \end{bmatrix}, \, \, \, l_2 = \begin{bmatrix}b_2 & \be_2 & h_1+f \id \\ 0 & -a_2 & -\alpha_1\end{bmatrix}, \, \, \, l_g = \begin{bmatrix}  \beta_g & h_{g-1}+(-1)^{g} f  \id \\  -a_g & -\alpha_{g-1} \\  0 & b_{g-1} \end{bmatrix},
$$
and
$$
l_i = \begin{bmatrix} b_i & \beta_i & h_{i-1}+(-1)^{i}f  \id \\ 0 & -a_i & -\alpha_{i-1} \\ 0 & 0 & b_{i-1} \end{bmatrix} \text{ if } 3 \le i \le g-1.
$$
Using Equation \eqref{eq:def-beta-1}, it follows that
\[
\Image l_1 = \mfb + (\al_{g-1}^*+ f a_g^*).
\]
All this remains true if we replace $f$ by $(-1)^g f$. Then Theorem \ref{app_a_and_b} shows that $I =
\mfb + (\al_{g-1}^*+ (-1)^g f a_g^*)$ is a Gorenstein ideal, and Diagram \eqref{last mapp.cone} yields    that  $I$ fits into an exact sequence
\begin{align*}\begin{CD}\label{ses}
 (\mfa/\mfb)(-d) @>\ffi>>  R/\mfb @>>> R/I  @>>>  0
 \end{CD}.
\end{align*}
It allows us to compute the Hilbert function of $\ker \ffi$.  Comparing with Corollary \ref{hilbert functions}, we deduce that the kernel of $\ffi$ is trivial. Hence, we obtain the desired short exact sequence
\begin{align*}
\begin{CD}
0 @>>> (\mfa/\mfb)(-d) @>\ffi>>  R/\mfb @>>> R/I  @>>>  0.
 \end{CD}
\end{align*}
Now it follows that the above complex $\mathbb{L}$ gives a free resolution of $I (d)$. Since $\beta_{g}$ is multiplication by a unit, we can split off the isomorphic free modules $A_g$ and $B_{g-1} (d)$ in the map $l_g$. After this  cancellation we get a  complex that  is, up to  a degree shift,  the claimed free resolution of $I$.
\end{proof}

The free resolution of $I$ we just derived is  smaller than the one obtained from the Complex  \eqref{lastMC}. In fact, it is often minimal.

\begin{corollary}
       \label{minimality}
If the polynomial $f$ is not a unit and each map $\al_i$ is minimal whenever $1 \le i \le g-1$, that is,
$\Image \al_i \subset \fm A_i$, then the resolution of $I$ described in Theorem \ref{KM construction thm}  is a graded minimal free resolution of $I$.
\end{corollary}

\begin{proof} Since the maps $\al_i$ are minimal, the definition of $\be_i$ (see Equation \eqref{eq:def beta}) implies that also $\be_i$ is a minimal map whenever $1 \le i \le g-1$. Now the description of the maps in the free resolution obtained in Theorem \ref{KM construction thm} shows that all its maps between free modules are minimal. Hence, it is a minimal resolution.
\end{proof}

The short exact sequence in Theorem~\ref{KM construction thm} allows us to re-interpret Theorem~\ref{app_a_and_b} in terms of liaison theory. To this end we recall the following definition.

Suppose $J\subset I\cap K$ are homogeneous ideals in $R$ with $\grade(I) = \grade(J) +1$, and $J$ is Cohen-Macaulay and generically Gorenstein. If there is an isomorphism of  graded $R$-modules
\[
I/J (-s) \cong K/J,
\]
then it is said that $K$ is obtained from $I$ by an {\em elementary biliaison} on $J$. It has the same grade as $I$.  (See \cite{KMMNP01,Mi, Har}  for more details.)

Using this concept, we get:

\begin{proposition}\label{biliaison}
The homogeneous Gorenstein ideal $I = (\mfb, \al^*_{g-1} + (-1)^g f a_g^*)$ in Theorem \ref{app_a_and_b} is obtained from $\mfa$ by an elementary biliaison on $\mfb$.
\end{proposition}

\begin{proof}
Theorem \ref{KM construction thm} provides the short exact sequence
$$\begin{CD}
  0 @>>> (\mfa/\mfb)(-d)@>\varphi>>   R/\mfb @>>> R/I @>>>  0
 \end{CD}.$$
 Thus, we get an isomorphism  $\mfa/\mfb(-d)\cong I/\mfb$. Since $\mfb$ is Gorenstein the claim follows directly from the definition of an elementary biliaison.
\end{proof}

So far we have studied the construction of a new homogeneous Gorenstein ideal $I$ of grade $g$ from smaller  homogeneous Gorenstein ideals $\mfb \subset \mfa$ of grades $g-1$ and $g$,
respectively. It is natural to ask when this construction can be reversed. One more precise version of this problem is whether, for  given homogeneous Gorenstein ideals $I$ and $\mfa$ of grade $g$, there is a homogeneous Gorenstein ideal
$\mfb \subset \mfa$ of grade $g-1$ such that $I$ can be obtained from  $\mfa$ by an elementary biliaison on $\mfb$. This question has already been considered in the local case in \cite{KM}. We now derive a necessary condition in the graded case. Recall that the Castelnuovo-Mumford regularity of a homogenous Gorenstein ideal $I \subset R$ is
\[
\reg I = \min \{m \; | \;  [H_{\mathfrak m}^i (I)]_{j} = 0 \text{ whenever } i+ j > m \},
\]
where $H_{\mathfrak m}^i (I)$ denotes the $i$-th local cohomology module with support in ${\mathfrak m}$. If $I$ has finite projective dimension over $R$, then its regularity can also be computed from a minimal free resolution as
\[
\reg I = \min \{m \; | \; [\Tor_i^R (I, R/{\mathfrak m})]_j = \text{ whenever } i - j > m \}.
\]

\begin{corollary}
     \label{cor:necc cond}
Let $I$ and $\mfa$ be homogeneous Gorenstein ideals of grade $g$.
If $\reg{I} - \reg{\mfa} $ is not even, then there is no homogeneous Gorenstein ideal $\mfb \subset \mfa$ such that $I$ can be obtained from  $\mfa$ by an elementary biliaison on $\mfb$.
\end{corollary}

\begin{proof}
From the degree shift of the last free module in the minimal free resolution of $\mfa$, we see (using the notation in Diagram \eqref{comp_BtoA}) that
\[
\reg \mfa  = v - g+1.
\]
If $I$ is obtained from $\mfa$ and $\mfb$ as in Theorem \ref{app_a_and_b}, then  the free resolution of $I$ described in Theorem~\ref{KM construction thm} gives
\[
\reg I = u + d - g + 1.
\]
It follows that
\[
\reg I - \reg \mfa = u- v + d = 2d.
\]
This implies the assertion.
\end{proof}

Using our description of the minimal free resolution in Theorem~\ref{KM construction thm}, we show in
Example~\ref{ex:generic 1551} below that there is a Gorenstein ideal $I$ that cannot be obtained by the construction in Theorem~\ref{app_a_and_b} if $\reg \mfa < \reg I$.
In general, it is open when a given Gorenstein ideal can be produced by an elementary biliaison as in Theorem \ref{app_a_and_b}.


\section{Examples}\label{Examples}

We describe various examples for the construction in Theorem \ref{app_a_and_b}. We conclude with exhibiting an ideal that can not be produced using this construction.

We begin with the easiest case, where $\mfa$ and $\mfb$ are complete intersection ideals. It extends Example \ref{ex:two-small-ci}. This case has also been discussed in the spirit of the original Kustin-Miller construction in \cite[Section 4]{Pa}.

\begin{example}
Let $R$ be a graded Gorenstein ring, and let $h_1,\ldots,h_g$ and $p_1,\ldots,p_{g-1}$ be regular sequences of homogeneous elements such that
\[
\mfb = (p_1,\ldots,p_{g-1}) \subset (h_1,\ldots,h_g) = \mfa.
\]
Then there is a homogeneous $g \times (g-1) $ matrix $M$ such that (as matrices)
\[
\begin{pmatrix}
p_1 & \ldots & p_{g-1}
\end{pmatrix} = \begin{pmatrix}
h_1 & \ldots & h_g
\end{pmatrix}  \cdot M.
\]
Setting $u = \sum \deg p_i$ and $v = \sum \deg h_j$, we get the following comparison map between the graded minimal free resolutions of $R/\mfa$ and $R/\mfb$
\[
\begin{CD}
   @.          0   @>>>        R(-u)  @>b_{g-1}>> .... @>>>  B_{1}  @>b_{1}>>   R  @>>> R/\mfb \\
  @.      @ VVV                     @VV\bigwedge^{g-1}MV            @.         @VVMV     @VV=V @VVV\\
  0 @>>> R(-v)  @>a_{g}>>   A_{g-1}        @>a_{g-1}>>  .... @>>>  A_{1}  @>a_{1}>>   R  @>>> R/\mfa
 \end{CD}
 \]
Denote by $M_i$ the square matrix obtained by deleting row $i$ of $M$ . Then, by Theorem~\ref{app_a_and_b}, for a sufficiently general $f \in R$ of degree $d = v-u \ge 0$, the ideal
\[
I = (p_1,\ldots,p_{g-1}, \det M_1 + f h_1,\ldots,\det M_g + f h_g)
\]
is a homogeneous Gorenstein ideal of grade $g$. Moreover, if no entry of the matrix $M$ is a unit, then the graded free resolution of $I$ described in Theorem~\ref{KM construction thm} is minimal. In particular, then $I$ has $2g-1$ minimal generators.

We can be more explicit in the following special case. Assume $x_1, x_2, \cdots , x_g$ is a regular sequence of homogeneous elements in $R$. Consider $\mathfrak{b} = (x_1^{m_1}, x_2^{m_2}, \cdots ,x_{g-1}^{m_{g-1}}) \subset (x_1^{n_1}, x_2^{n_2}, \cdots , x_g^{n_g}) = \mfa$ and assume $d := \sum\limits_{i=1}^{g-1}m_i - \sum\limits_{i=1}^{g}n_i \geq 0.$ Then, for a sufficiently general $f \in R$ of degree $d$,
\[
I = (x_1^{m_{1}},\cdots , x_{g-1}^{m_{g-1}} , fx_1^{n_{1}} , \cdots ,fx_{g-1}^{n_{g-1}},   c + fx_g^{n_{g}})
\]
is a Gorenstein ideal, where $c = \prod\limits_{j=1}^{g-1} x_j^{m_j-n_j}$. Moreover, if $m_j > n_j$ for each $j = 1,\ldots,g-1$, then the resolution in Theorem~\ref{KM construction thm} is a minimal free resolution of $I$.
\end{example}

In the next example we show that all the Gorenstein ideals with socle degree two can be obtained by one elementary biliaison from a complete intersection.

\begin{example}
    \label{exa:sally ideal}
Consider the Artinian Gorenstein ideals $I \subset R = K[x_1,\ldots,x_n]$ with $h$-vector $(1, n, 1)$, where $K$ is a field. These ideals have been classified by Sally in \cite[Theorem 1.1]{sally}. Each such ideal is of the form
\[
I = (x_i x_j \; | \; 1 \le i < j \le n) + (x_1^2 - c_1 x_n^2, \ldots, x_{n-1}^2 - c_{n-1} x_n^2),
\]
where $c_1, \ldots,c_{n-1} \in K$ are suitable units. It can be obtained by an elementary biliaison as in Theorem \ref{app_a_and_b} from $\mfa = (x_1,\ldots,x_n)$ on $\mfb R$, where $\mfb$ is such a Sally ideal in $n-1$ variables. More precisely, define the ideal $\mfb$ as
\[
\mfb = (x_i x_j \; | \; 1 \le i < j \le n-1) + (x_1^2 - c_1 x_n^2, \ldots, x_{n-2}^2 - c_{n-2} x_n^2).
\]
Then it is not too difficult to see that there are the following links
\[
\mfa \linkby[(\mfb, x_n)] (\mfb, x_n, x_{n-1}^2) \linkby[(\mfb, x_{n-1}^2 - c_{n-1} x_n^2)] I.
\]
Note that $(\mfb, x_n, x_{n-1}^2) = (x_1,\ldots,x_{n-1})^2 + (x_n)$.
\end{example}

The following classical example has been studied from various points of view.

\begin{example}
Let $M = (x_{i j})$ be a generic $n \times n$ matrix, where $n \ge 2$.  The ideal $I = \text{I}_{n-1}(M)$ in $K[M]$, generated by the submaximal minors of $M$ is a Gorenstein ideal of grade four. Its graded  minimal free resolution is given by the Gulliksen-Neg\.{a}rd complex (see \cite{GN}):
\[
0 \to R (-2n)\to R^{n^2}(-n-1) \to R^{2(n^2-1)}(-n) \to R^{n^2}(-n+1) \to I \to 0.
\]
Kustin and Miller show that this resolution can be obtained by using their original construction (see \cite[Example 2.4]{KM}). Gorla \cite{Gor} studies these ideals from a liaison-theoretic point of view. Here we make the linkage steps more explicit.

If $n = 2$, then $I$ is a complete intersection. Assume $n \ge 3$, and let $N$ be the generic  $(n-1)\times (n-1)$ obtained from $M$ by deleting its last row and column. Its $(n-2)\times (n-2)$ minors generate a homogeneous Gorenstein ideal $\mfa = \text{I}_{n-2}(N)$ of grade $4$. Denote by $M_{i, j}$ the $(n-1)\times (n-1)$ minor of $M$ obtained by deleting row $i$ and column $j$. The ideal
\[
\mfb = (M_{1,n}, M_{2, n}, \cdots,M_{n-1, n}, M_{n,1}, \cdots, M_{n, n-1})
\]
is a Gorenstein ideal of grade three (see, e.g, \cite[Example 2.4]{KM}). Sylvester's identity implies that  (see, the proof of Theorem 3.1 in \cite{Gor}):
\[
N_{1, 1} \cdot I + \mfb = M_{1, 1} \cdot \mfa + \mfb.
\]
It follows that there are the following links
\[
\mfa \linkby[(\mfb, N_{1,1})] (\mfb, N_{1,1}, M_{1,1}) \linkby[(\mfb, M_{1,1})] I.
\]
Hence $I$ can be obtained from $\mfa$ by an ascending biliaison on $\mfb$ as described in Theorem~\ref{app_a_and_b}. Repeating the construction, we see that $I$ can be obtained from the complete intersection $(x_{1 1}, x_{1 2}, x_{2 1}, x_{2 2})$ by $(n-2)$ such ascending biliaisons.
\end{example}

Now we consider some Gorenstein ideals with 9 generators and 16 syzygies. Such Gorenstein ideals are investigated in depth from the point of view of unprojections  in \cite{BKR}.

\begin{example}
Let $R =K[a,b,c,d,e,f,x,y,z]$ be a polynomial ring in 9 variables over a field $K$. Consider  a generic $3\times3$ symmetric matrix $A$ and  a generic  skew-symmetric matrix $B$:
\[
A  =  \begin{bmatrix} a & b & c \\ b & d & e \\ c & e & f \end{bmatrix} \text{  and }  B = \begin{bmatrix} 0 & x & y \\ -x & 0 & z \\ -y & -z & 0 \end{bmatrix}.
\]
Then, for  $\la \neq 0$ in $K$, define a $6\times 6$ skew-symmetric matrix $N =  \begin{bmatrix} B & A \\ -A & \la B \end{bmatrix}.$ It is called
 ``extrasymmetric''  in \cite{BKR, BRS} because it is obtained from a generic skew-symmetric matrix by specializing some of the variables. The ideal $\mfa$ generated by the $4\times4$ Pfaffians of $N$ is a homogeneous Gorenstein ideal of grade $4$:
\begin{align*}
\mfa =  &(b^2-ad+\la x^2 , bc-ae+\la xy, c^2-af+\la y^2 , cd-be+\la xz, ce-bf+\la yz, \\
 & \hspace*{.2cm} e^2-df+\la z^2, cx-by+az, ex-dy+bz, fx-ey+cz).
\end{align*}
It is the defining ideal of the Segre embedding
of $\mathbb{P}^2\times \mathbb{P}^2$ into $\mathbb{P}^8$ and  a typical case of a Tom unprojection (see \cite{BKR,BRS,PR}). In particular, $\mfa$ is equal to the ideal generated by the
$2\times2$ minors of a $3\times3$ generic matrix $A + \sqrt{-\la}B$. Hence, the  Gulliksen and Neg\.{a}rd complex gives its minimal free resolution:
\[
\begin{CD} 0 @>>> R(-6) @>a_4>> R^9(-4) @>a_3>> R^{16}(-3) @>a_2>> R^9(-2)@>a_1>> \mfa @>>> 0.\end{CD}
\]
In order to perform the construction of Theorem \ref{app_a_and_b}, we choose the first three listed generators of $\mfa$ to define a complete intersection
$$\mfb = (b^2-ad+\la x^2 , bc-ae+\la xy, c^2-af+\la y^2 )$$
inside $\mfa$. Then we link as follows:
\[
\mfa \linkby[(\mfb,cd-be+\la xz)]  (\mfb, cd-be+\la xz, ax) \linkby[(\mfb, ax+(cd-be+\la xz))]  I.
\]
Explicitly, the resulting ideal $I$ is
\begin{align*}I =  & (e^2-df-cx+by+az+\la z^2, ce-bf+ay+\la yz, cd-be+ax+\la xz, \\
    & \hspace*{.2cm} c^2-af+\la y^2,  bc-ae+\la xy, ac+\la fx-\la ey+\la cz, b^2-ad+\la x^2, \\
    & \hspace*{.2cm}  ab+\la ex-\la dy+ \la bz, a^2+\la cx-\la by+\la az).
\end{align*}
It has the same Betti table as $\mfa$. In fact, $I$ is again an example of  a Tom unprojection. This time  the extrasymmetric matrix is
$$M = \begin{bmatrix} 0 & x & y & a & b & c \\
                 -x & 0 & \frac{1}{\la}a+z & b & d & e \\
                  y & -\frac{1}{\la}a-z & 0 & c & e & f \\
                 -a & -b & -c & 0 & \la x & \la y \\
                 -b & -d & -e & -\la x & 0 & a+\la z \\
                 -c & -e & -f & -\la y & -a-\la z & 0
\end{bmatrix},$$
so $I = \pf_4(M)$.
\end{example}

We conclude with an example of a Gorenstein ideal that cannot be produced using the construction of Theorem~\ref{app_a_and_b} with a strictly ascending biliaison.

\begin{example}
    \label{ex:generic 1551}
Let $I$ be a generic Artinian Gorenstein ideal in  $R = K[x_1,\ldots,x_5]$ with $h$-vector $(1, 5, 5, 1)$, where $K$ is an infinite field. It has the least possible Betti numbers. More precisely, its  graded minimal free resolution is pure and has the form
\begin{equation}
    \label{eq:1551-res}
0 \to R(-8) \to R^{10} (-6) \to R^{16} (-5) \to R^{16} (-3) \to R^{10} (-2) \to I \to 0.
\end{equation}
We claim that there are no Gorenstein ideals $\mfa$ and $\mfb$ to produce $I$ using a biliaison as in
Theorem~\ref{app_a_and_b} that is strictly ascending, i.e., $d > 0$ or, equivalently, $\mfa$ has smaller regularity than $I$.

Indeed, to see this assume such ideals $\mfa$ and $\mfb$ do exist. Since $\reg I = 3$, this forces $\reg \mfa = 1$ by Corollary~\ref{cor:necc cond}. It follows that the $h$-vector of $R/\mfa$ must be $(1, 1)$. Hence, possible after a change of coordinates, we may assume
\[
\mfa = (x_1,x_2, x_3, x_4, x_5^2).
\]
Thus, Corollary~\ref{hilbert functions} gives that $R/\mfb$ has $h$-vector $(1, 4, 1)$, that is, $\mfb$ is a Sally ideal (see Example~\ref{exa:sally ideal}). Its graded minimal free resolution has the form
\[
0 \to R(-7) \to R^6 (-5) \to R^5 (-3) \oplus R^5 (-4) \to R^6(-2) \to \mfb \to 0.
\]
By Theorem~\ref{KM construction thm}, we have the following short exact sequence of graded $R$-modules
$$\begin{CD}
  0 @>>> (\mfa/\mfb)(-1)@>>>   R/\mfb @>>> R/I @>>>  0.
 \end{CD}$$
Consider now the comparison map between the resolutions of $(\mfa/\mfb)(-1)$ and $R/\mfb$ in homological degree two. Using the notation of the proof of Theorem~\ref{KM construction thm}, it is
\[
\begin{CD}
A_3 (-1) \oplus B_2 (-1) @>[\beta_{3}, h_2 - f \id]>> B_2 = R^5 (-4) \oplus R^5 (-5)
\end{CD}
\]
Since $\deg f = 1$, the map $h_2 - f \id$ is minimal. Moreover, notice that $A_3 (-1) = R^4 (-4) \oplus R^6 (-5)$. Considering the map $\beta_3$ in degree 4,  the mapping cone procedure implies that $[\Tor_2^R (R/I, K)]_4 \neq 0$. Hence $I$ does not have a pure resolution as in \eqref{eq:1551-res}, which completes the argument.
\end{example}

%
%


\end{document}